\documentclass[11pt]{amsart}
\usepackage[dvips,pdftex]{graphicx,color}
\usepackage{amssymb,amscd,latexsym,epsfig,xypic,hyperref}
\usepackage[mathscr]{euscript}

\DeclareFontFamily{OT1}{rsfs}{}
\DeclareFontShape{OT1}{rsfs}{n}{it}{<-> rsfs10}{}
\DeclareMathAlphabet{\curly}{OT1}{rsfs}{n}{it}

\makeatletter
\newcommand{\eqnum}{\refstepcounter{equation}\textup{\tagform@{\theequation}}}
\makeatother

\renewcommand\;{\hspace{.6pt}}

\newcommand\PP{\mathbb P}
\newcommand\LL{\mathbb L}
\newcommand\C{\mathbb C}

\newcommand\Z{\mathbb Z}

\newcommand\cO{\mathcal O}

\renewcommand\t{\mathfrak t}

\renewcommand\({\big(}
\renewcommand\){\big)}
\makeatletter
\newcommand{\so}{\ \ext@arrow 0359\Rightarrowfill@{}{\hspace{3mm}}\ }
\makeatother
\newcommand{\rt}[1]{\xrightarrow{\ #1\ }}
\newcommand\To{\longrightarrow}

\newcommand\into{\hookrightarrow}
\newcommand\INTO{\ \ar@{^(->}[r]<-.2ex>}
\newcommand{\Into}{\ensuremath{\lhook\joinrel\relbar\joinrel\rightarrow}}

\renewcommand\_{^{}_}
\newcommand\take{\backslash}

\newfont{\bigtimesfont}{cmsy10 scaled \magstep5}
\newcommand{\bigtimes}{\mathop{\lower0.9ex\hbox{\bigtimesfont\symbol2}}}
\renewcommand\={\ =\ }

\newcommand\udot{^{\bullet}}

\newcommand\rk{\operatorname{rank}}
\newcommand\vir{\operatorname{vir}}

\newcommand\vd{\operatorname{vd}}

\newcommand\id{\operatorname{id}}

\newcommand\Proj{\operatorname{Proj}\;}
\newcommand\Spec{\operatorname{Spec}\;}

\newcommand\Sym{\operatorname{Sym}}

\newcommand\beq[1]{\begin{equation}\label{#1}}
\newcommand\eeq{\end{equation}}
\newcommand\beqa{\begin{eqnarray*}}
\newcommand\eeqa{\end{eqnarray*}}

\newcommand\arXiv[1]{\href{http://arxiv.org/abs/#1}{arXiv:#1}}
\newcommand\mathAG[1]{\href{http://arxiv.org/abs/math/#1}{math.AG/#1}}

\DeclareRobustCommand{\SkipTocEntry}[3]{}
\makeatletter
\newcommand\@dotsep{4.5}
\def\@tocline#1#2#3#4#5#6#7{\relax
  \ifnum #1>\c@tocdepth 
  \else
    \par \addpenalty\@secpenalty\addvspace{#2}%
    \begingroup \hyphenpenalty\@M
    \@ifempty{#4}{%
      \@tempdima\csname r@tocindent\number#1\endcsname\relax
    }{%
      \@tempdima#4\relax
    }%
    \parindent\z@ \leftskip#3\relax \advance\leftskip\@tempdima\relax
    \rightskip\@pnumwidth plus1em \parfillskip-\@pnumwidth
    #5\leavevmode #6\relax
    \leaders\hbox{$\m@th
      \mkern \@dotsep mu\hbox{.}\mkern \@dotsep mu$}\hfill
    \hbox to\@pnumwidth{\@tocpagenum{#7}}\par
    \nobreak
    \endgroup
  \fi}
\makeatother

\makeatletter \@addtoreset{equation}{section} \makeatother
\renewcommand{\theequation}{\thesection.\arabic{equation}}

\newtheorem{thm}[equation]{Theorem}
\newtheorem*{thm*}{Theorem}
\newtheorem{lem}[equation]{Lemma}
\newtheorem{cor}[equation]{Corollary}
\newtheorem{prop}[equation]{Proposition}

\newtheorem*{defthm*}{Definition/Theorem}

\newtheorem{rmk}[equation]{Remark}

\title{A K-theoretic Fulton class\vspace{-2mm}}
\author{Richard P. Thomas\vspace{-3mm}}

\begin{document}
\maketitle
%


\thispagestyle{empty}
\section{Summary}
Fix a quasi-projective scheme $M$ over the complex numbers, and pick a global embedding in a smooth ambient variety $A$. Let $I\subset\cO_A$ denote the ideal sheaf of $M$. We get the cone on the embedding $M\into A$,
\beq{cone}
C_MA\ :=\ \Spec\,\bigoplus\nolimits_{i\ge0}\,I^i\big/I^{i+1}.
\eeq
Then Fulton's total Chern class of $M$ \cite[Example 4.2.6]{Fu} is defined to be
\beq{FHstar}
c\_F(M)\ :=\ c\(T_A|_M\)\cap s(C_MA)\ \in\ A_\bullet(M),
\eeq
where $c$ is the total Chern class and $s$ denotes the Segre class. The result is independent of the choice of embedding. When $M$ is smooth, $c\_F(M)$ is just the total Chern class $c(T_M)\cap[M]=\sum_{i\ge0}c_i(M)\cap[M]$ of $M$.

We define a K-theoretic analogue. For notation see Section \ref{kFul}; in particular $t$ denotes the class of the weight one irreducible representation of $\C^*$.
\begin{defthm*}
Let $\C^*$ act trivially on $M$, with weight $1$ on $\Omega_A|_M$, and with weight $i$ on $I^i/I^{i+1}$. The K-theoretic Fulton class
$$
\qquad\Lambda_M\ :=\ \Lambda\udot\;\Omega_A\big|_M\otimes\Big(\bigoplus\nolimits_{i\ge0}\,I^i\big/I^{i+1}\Big)\ \in\ K_0(M)[\![t]\!],
$$
is independent of the smooth ambient space $A\supset M$ and is \emph{polynomial in $t$,}
$$
\Lambda_M\=\sum_{i=0}^d(-1)^i\Lambda^i_Mt^i\ \in\ K_0(M)[t],
$$
where $d$ is the embedding dimension of $M$. When $M$ is smooth, $\Lambda^i_M=\Omega^i_M$.
\end{defthm*}

\noindent In Proposition \ref{ddR} we describe $\Lambda_M$ in de Rham terms for $M$ lci, but for more general $M$ we have not seen these classes in the literature. \medskip

Suppose $M$ has a perfect obstruction theory $E\udot\to\LL_M$ \cite{BF}  of virtual dimension $\vd:=\rk(E\udot)$. Then we get a virtual cycle \cite{BF} for which Siebert \cite{Sie} gave the following formula in terms of the Fulton class
\beq{bernd}
[M]^{\vir}\=\left[s\((E\udot)^\vee\)\cap c\_F(M)\right]_{\vd}\ \in\ A_{\vd}(M).
\eeq
A K-theoretic analogue is the following.

\begin{thm*} Given a perfect obstruction theory $E\udot\to\LL_M$ the virtual structure sheaf can be calculated in terms of the K-theoretic Fulton class $\Lambda_M$ as
\beq{back}
\cO^{\vir}_M\=\left[\frac{\Lambda_M}{\Lambda\udot(E\udot)}\right]_{t=1}\=\left[\frac{\Lambda_M}{\Lambda\udot\;\LL_M^{\!\vir}}\right]_{t=1}.
\eeq
\end{thm*}

Siebert's formula \eqref{bernd} showed that $[M]^{\vir}$ depends only on the scheme structure of $M$ and the K-theory class $[E\udot]$ of the virtual cotangent bundle $\LL_M^{\!\vir}:=E\udot$ (and not the map $E\udot\to\LL_M$ defining the perfect obstruction theory). Similarly the above Theorem implies that $\cO^{\vir}_M$ also depends only on $M$ and $[E\udot]$. One aspect of the analogy we refer to is that --- with some work --- \eqref{bernd} follows from \eqref{back} by the virtual Riemann-Roch theorem of \cite{CFK,FG}. For further discussion of the analogy between cohomological and K-theoretic Chern classes and Fulton classes, see Section \ref{KCh}.
\medskip

To understand why $\Lambda_M$ might be polynomial in $t$, consider what happens in the case that $M$ is a zero dimensional Artinian scheme. (The general case is a family version of this.) By \cite[Section 11]{AM} the Hilbert series of the graded $\cO_M$-module $\bigoplus_{i\ge0}I^i/I^{i+1}$ is $p(t)/(1-t)^d$, where $d=\dim A$ and $p$ is polynomial in $t$. But $\Omega_A|_M\cong\cO_M^{\oplus d}\otimes t$ so tensoring by $\Lambda\udot\;\Omega_A|_M=(1-t)^d\cO_M$ gives the result. \medskip

\noindent\textbf{Acknowledgements.} It is an honour to dedicate this paper to Bill Fulton on the occasion of his 80th birthday. Much of my career has depended crucially on the masterpiece \cite{Fu}. You should never meet your heroes, but I did and it was a genuine pleasure.

Thanks to Ben Antieau for suggesting Proposition \ref{ddR} and Dave Anderson for lengthy discussions on the link to K-theoretic Chern and Segre classes (see Section \ref{KCh}). I am especially grateful to an alert referee, Jason Starr and Will Sawin for correcting my false assumptions about embedding dimension \cite{SS}. I acknowledge support from EPSRC grant EP/R013349/1.

\section{K-theoretic analogue of Fulton's Chern class}
\label{kFul}
Throughout this paper we fix a quasi-projective scheme $M$ over $\C$, endow it with the \emph{trivial} $\C^*$ action, and work with the $\C^*$-equivariant K-theory of coherent sheaves on $M$. This is
$$
K_0(M)^{\C^*}\ \cong\ K_0(M)[t,t^{-1}]
$$
as a module over $K(\mathrm{pt})^{\C^*}=\Z[t,t^{-1}]$, where $t$ is the class of the weight one irreducible representation of $\C^*$. In fact we only ever use only the subgroup $K_0(M)[t]$ generated by coherent sheaves with $\C^*$ actions \emph{with nonnegative weights}, and its completion
$$
K_0^{\C^*}\!(M)\_{\ge0}\otimes\_{\Z[t]}\Z[\![t]\!]\=K_0(M)[\![t]\!].
$$
For $E$ locally free we use $\Lambda\udot E$ to refer to the K-theory class $\sum_{i=0}^{\rk E}(-1)^i\Lambda^iE$. If $E$ carries a $\C^*$ action then on the total space of $\pi\colon E^*\to M$ the tautological section of $\pi^*E^*$ is $\C^*$-equivariant, giving an equivariant Koszul resolution $\Lambda\udot\pi^*E\rt\sim\cO_M$ and so the identity $\Lambda\udot\pi^*E=\cO_M$ in $K_0(E^*)^{\C^*}$. If $E$ has only strictly positive weights we can apply $\pi_*\colon K_0(E^*)^{\C^*}\to K_0(M)[\![t]\!]$ to this, giving $\Lambda\udot E\otimes\Sym\udot\!E=\cO_M$. That is, $\Lambda\udot E=1/\Sym\udot E$ in $K_0(M)[\![t]\!]$.

\begin{rmk} The results here commute with operations such as localisation with respect to a nontrivial $\C^*$ action on $M$ (as used in \cite{Th}, for instance) since any such $\C^*$ action commutes with the trivial  $\C^*$ action used here.
\end{rmk}

Let $\C^*$ act on $\Omega_A|_M$ with weight $1$ and on $I^i/I^{i+1}$ with weight $i$. Consider
\beq{FuK}
\qquad\Lambda_M\ :=\ \Lambda\udot\;\Omega_A\big|_M\otimes\Big(\bigoplus\nolimits_{i\ge0}\,I^i\big/I^{i+1}\Big)\ \in\ K_0(M)[\![t]\!].
\eeq

\begin{thm}\label{thm}
This $\Lambda_M$ is independent of the smooth ambient space $A\supset M$ and is \emph{polynomial in $t$,} defining the K-theoretic Fulton classes
$$
\Lambda_M\=\sum_{i=0}^d(-1)^i\Lambda^i_Mt^i\ \in\ K_0(M)[t].
$$
Here $d$ is the embedding dimension of $M$. When $M$ is smooth, $\Lambda^i_M=\Omega^i_M$.
\end{thm}

\noindent\textbf{Remark.} By ``embedding dimension" we mean the smallest dimension of a smooth variety $A$ containing $M$. It is natural to wonder if we can replace this in the theorem by the maximum of the dimensions of the Zariski tangent spaces $T_xM$ over all closed points $x\in M$. This number is in general smaller than the embedding dimension \cite{SS}.
\medskip

\begin{proof}[Proof of Theorem]
Fix two embeddings $M\subset A_i,\ i=1,2$ with ideal sheaves $I_i$ and cones $C_i:=C_MA_i$. We get the induced diagonal inclusion
$$
M\ \subset\ A_1\times A_2
$$
with ideal $I_{12}$ and cone $C_{12}:=C_M(A_1\times A_2)$.
This gives the exact sequence of cones \cite[Example 4.2.6]{Fu},
\beq{FuSES}
0\To T_{A_2}\big|_M\To C_{12}\To C_1\To0.
\eeq
That is, $I_{12}^i/I_{12}^{i+1}$ has an increasing filtration beginning with $I_1^i/I_1^{i+1}$ and with graded pieces $\Sym^j\;\Omega_{A_2}|_M\otimes I_1^{i-j}/I_1^{i-j+1}$. Therefore,
in completed equivariant K-theory,
\begin{align*}
\Lambda\udot\Omega_{A_1\times A_2}\big|_M&\otimes\Big(\bigoplus\nolimits_{i\ge0}\,I_{12}^i\big/I_{12}^{i+1}\Big) \\
&=\ \Lambda\udot\Omega_{A_1\times A_2}\big|_M\otimes\Big(\bigoplus_{i,j\ge0}\Sym^j\Omega_{A_2}\big|_M\otimes I_1^{i-j}\big/I_1^{i-j+1}\Big) \\ &=\ 
\Lambda\udot\Omega_{A_1}\big|_M\otimes\Lambda\udot\Omega_{A_2}\big|_M\otimes\Sym\udot\Omega_{A_2}\big|_M\otimes\bigoplus_{i\ge0}I_1^i\big/I_1^{i+1} \\ &=\
\Lambda\udot\Omega_{A_1}\big|_M\otimes\bigoplus\nolimits_{i\ge0}\,I_1^i\big/I_1^{i+1}.
\end{align*}

This gives the independence from $A$. We now show \eqref{FuK} is polynomial in $t$ of degree $\le d:=\dim A$.

Writing \eqref{FuK} as $\pi_*\pi^*\Lambda\udot\;\Omega_A|_M$, where 
$\pi\colon C_MA\to M$ denotes the projection, the power series in $t$ comes from the fact that $\pi$ is not proper. So we only need to show that $\pi^*\Lambda\udot\;\Omega_A|_M$ is equivalent in $K^{\C^*}_0(C_AM)$ to a class pushed forward from $M$ with $\C^*$ weights in $[0,d]$.

The basic idea of the proof is the following. Suppose we could pick a $\C^*$-invariant section $s$ of $\pi^*T_A|_M$ which --- on the complement $C_MA\take M$ of $M$ --- has vanishing locus of the expected dimension 0. By $\C^*$ invariance this locus is empty, so the vanishing locus of $s$ is just $M$ (on which $s$ must be zero since $T_A|_M$ has weight $-1$). Therefore the Koszul complex $\pi^*\Lambda\udot\;\Omega_A|_M$ is exact on $C_MA\take M$. It is equivalent in K-theory to its cohomology sheaves which are supported on $M\subset C_MA$ and have $\C^*$ weights in $[0,d]$. This gives the result required.

In general there only exist sections of $\pi^*T_A|_M$ \emph{with poles} so we have to twist by a divisor $D$ in the proof below, complicating matters. The reader is encouraged to read it first in the no-poles case --- where $i=0$, there are no $\cO_D$ terms in \eqref{codimm}, and everything reduces to the Koszul complex \eqref{kzz}. The general case involves the generalised Koszul complexes \eqref{kz2}.

So pick a line bundle $L\gg0$ on $M$ and $D\in|L|$. We work with
$$
E\ :=\ T_A\big|_M\otimes L
$$
since it has sections which we can use to cut down to lower dimensions. We give $L$ the trivial $\C^*$ action of weight 0 so that $E$ has weight $-1$. We have
$$
\Lambda^k\Omega_A\big|_M\=\Lambda^kE^*\otimes L^k\=\Lambda^kE^*\otimes(\cO-\cO_D)^{-k}.
$$
Using $\otimes$ for the \emph{derived} tensor product, in K-theory we deduce\footnote{When $k=0=i$ we have to work with the standard negative binomial convention that ${k+i-1\choose i}={-1\choose0}=1$. Therefore we also set $\Sym^{-1}\cO^{i+1}$ to be $\cO$ for $i=0$ and 0 for $i>0$.\label{foot}}
\begin{eqnarray} \nonumber
\Lambda\udot\;\Omega_A\big|_M &=& \sum_{k=0}^d(-1)^k\Lambda^kE^*\otimes\sum_{i=0}^d{k+i-1\choose i}\cO_D^{\otimes i} \\
&=& \sum_{i=0}^d\cO_D^{\otimes i}\otimes\sum_{k=0}^d(-1)^k\Lambda^kE^*\otimes\Sym^{k-1}\cO^{\;i+1}. \label{codimm}
\end{eqnarray}
Here we have used $[\cO_D^{\otimes i}]=0$ for $i>d\ge\dim M$. For $L\gg0$ basepoint free this follows from taking divisors $D_1,\ldots,D_{d+1}\in|L|$ with empty intersection.

More generally, fix any $i\ge0$. Since the sections $\pi^*H^0(L)$ of $\pi^*L$ are basepoint free, we may pick generic divisors
$D_1,\ldots,D_i$ such that their intersection $D^{(i)}:=D_1\cap\dots\cap D_i$, and the intersection $\pi^*D^{(i)}$ of their pullbacks $\pi^*D_i$, are \emph{transverse}. That is, we have an equality of \emph{derived} tensor products
$$
\cO_{D_1}\otimes\dots\otimes\cO_{D_i}\=\cO_{D^{(i)}} \quad\text{and}\quad
\cO_{\pi^*D_1}\otimes\dots\otimes\cO_{\pi^*D_i}\=\cO_{\pi^*D^{(i)}}.
$$
Therefore by \eqref{codimm} we can write $\pi^*\Lambda\udot\Omega_A|_M$ as the sum over $i=0,\dots,d$ of
\beq{thing}
\sum_{k=0}^d(-1)^k\pi^*\Lambda^kE^*\big|_{D^{(i)}}\otimes\Sym^{k-1}\cO^{\;i+1}_{\pi^*D^{(i)}}.
\eeq
So we are left with showing \eqref{thing} is equal, in equivariant K-theory, to a class pushed forward from $D^{(i)}\subset\pi^*D^{(i)}$ with $\C^*$ weights in $[0,d]$.

Consider the projectivised cone $\PP(C_MA)=\Proj\,\bigoplus_{j\ge0}I^j/I^{j+1}$ and its fibre product
\beq{fib}
\PP(C_MA)\times\_MD^{(i)}\rt qD^{(i)}.
\eeq
It has dimension $d-i-1$, while $q^*E(1)$ has rank $d$ and is globally generated. Therefore $q^*E(1)$ has $i+1$ linearly independent sections on \eqref{fib} by a standard argument (see \cite[page 148]{Mu}, for instance.)

Equivalently, $\pi^*\(E|\_{D^{(i)}}\)$ has $i+1$ $\C^*$-equivariant sections which are  linearly independent away from the zero section $D^{(i)}\subset\pi^*D^{(i)}$ and whose scheme-theoretic degeneracy locus is precisely $D^{(i)}\subset\pi^*D^{(i)}$ (i.e. its 0\;th Fitting ideal is the irrelevant ideal $\bigoplus_{j>0}I^j/I^{j+1}\subset\cO_{\pi^*D^{(i)}}$).

Therefore, taking $i=0$ to begin with, we get one equivariant section of $\pi^*E$ and a Koszul complex
\beq{kzz}
\pi^*\Lambda^d E^*\To\pi^*\Lambda^{d-1}E^*\To\dots\To\pi^*E^*\To\cO_{C_MA}
\eeq
with cokernel $h^0\cong\cO_M$. Considering it as a sheaf of dgas, its other cohomology sheaves are modules over $h^0$, so they are also supported scheme-theoretically on $M\subset C_MA$. Since $M$ is $\C^*$-fixed, and $\C^*$ acts on the $\Lambda^kE^*$ term of \eqref{kzz} with weight $k\in[0,d]$, this shows the (pushdown to $M$ of) the cohomology sheaves of \eqref{kzz} have $\C^*$ weights in $[0,d]$. Therefore, for $i=0$, its K-theory class \eqref{thing} has $\C^*$ weights in $[0,d]$.

For $i\ge1$ we use the generalised Koszul complex
\begin{multline}\label{kz2}
\pi^*\Lambda^dE^*|\_{D^{(i)}}\otimes\Sym^{d-1}\cO_{\pi^*D^{(i)}}^{\;i+1}\To\pi^*\Lambda^{d-1}E^*|\_{D^{(i)}}\otimes\Sym^{d-2}\cO_{\pi^*D^{(i)}}^{\;i+1} \\
\To\cdots\To\pi^*\Lambda^2E^*|\_{D^{(i)}}\otimes\cO_{D^{(i)}}^{\;i+1}\To\pi^*E^*\big|_{D^{(i)}}.
\end{multline}
It is exact away from $D^{(i)}\subset\pi^*D^{(i)}$, since its twist by $\pi^*\Lambda^dE\big|_{D^{(i)}}$ is the obvious resolution of $\Lambda^{d-1}Q$, where $Q$ is the locally free quotient $0\to\cO^{i+1}\to\pi^*E|\_{D^{(i)}}\to Q\to0$; since rank$\,Q=d-i-1<d-1$, this is zero. Furthermore, its cohomology sheaves are well-known\footnote{Nonetheless, I couldn't find a decent reference --- there's something terribly wrong with the algebra literature. On $\rho\colon\PP^i\times\pi^*D^{(i)}\to\pi^*D^{(i)}$ the $i+1$ sections of $\pi^*E$ give a section of $\rho^*\pi^*E(1)$ with scheme-theoretic zero locus $\PP^i\times D^{(i)}\subset\PP^i\times\pi^*D^{(i)}$. Therefore the Koszul complex $\Lambda\udot(\rho^*\pi^*E^*(-1))$ of this section has cohomology sheaves supported on $\PP^i\times D^{(i)}$. Applying $R\rho_*\(\ \cdot\ \otimes\cO(-i)\)[i]$ gives \eqref{kz2} and the result claimed.} to be supported scheme-theoretically on $D^{(i)}$.

Since $D^{(i)}$ is $\C^*$-fixed, and $\C^*$ acts on the $\Lambda^kE^*$ term of \eqref{kz2} with weight $k$ for $k=0,1,\ldots,d$, this shows the (pushdown to $D^{(i)}$ of) the cohomology sheaves of \eqref{kz2} have $\C^*$ weights in $[0,d]$. Therefore the K-theory class \eqref{thing} of \eqref{kz2} has $\C^*$ weights in $[0,d]$.
\end{proof}

If we pick a locally free resolution $F^*\rt{s}I$ on $A$ (i.e. a vector bundle $F\to A$ with a section $s$ cutting out $s^{-1}(0)=M\subset A$) then we can express $\Lambda_M$ differently as follows. Give $F$ the $\C^*$ action of weight $-1$, so the embedding
\beq{coneinto}
C_MA\ \Into\ F|_M
\eeq
induced by $s$ is equivariant. Let $\iota\colon M\into F|_M$ and $\pi\colon F|_M\to M$ denote the zero section and projection respectively. 

\begin{lem}
The K-theoretic Fulton class equals
$$
\Lambda_M\=L\iota^*\cO\_{C_MA}\otimes\left.\frac{\Lambda\udot\Omega_A}{\Lambda\udot F^*}\right|_M.
$$
\end{lem}

\begin{proof}
Applying $\pi_*\iota_*=\id$ to the right hand side gives
$$
\pi_*\left(\iota_*\cO_M\stackrel L\otimes\cO\_{C_MA}\right)\otimes\left.\frac{\Lambda\udot\Omega_A}{\Lambda\udot F^*}\right|_M.
$$
Then $\iota_*\;\cO_M=\pi^*\Lambda\udot F^*|_M$ by the Koszul resolution of the zero section $M\into F|_M$, so by \eqref{cone} we get
\[
\Lambda_M\=\Big(\bigoplus\nolimits_{i\ge0}\,I^i/I^{i+1}\Big)\otimes\Lambda\udot\Omega_A\big|_M\,. \qedhere
\]
\end{proof}

\section{de Rham cohomology}
Consider the pushforward of $\Lambda_M$ to the formal completion $\widehat A$ of $A$ along $M$. Its K-theory class looks remarkably similar to that of Hartshorne's algebraic de Rham complex \cite{Ha}
\beq{ahat}
\big(\Lambda\udot\;\Omega_{\widehat A},d\big).
\eeq
If we discard the de Rham differential in \eqref{ahat} and filter by order of vanishing along $M$ (with $n$th filtered piece $I^n\otimes\Lambda\udot\;\Omega_{\widehat A}$) then the associated graded is
\beq{agr}
\Lambda\udot\;\Omega_A\big|_M\otimes\Big(\bigoplus\nolimits_{i\ge0}\,I^i\big/I^{i+1}\Big).
\eeq
This is just (the push forward from $M$ to $\widehat A$ of) $\Lambda_M$ with its $\C^*$ action forgotten. However convergence issues stop us from equating \eqref{agr} with \eqref{ahat} in K-theory. Putting the $\C^*$ action back into \eqref{agr} we get convergence to $\Lambda_M$ in completed equivariant K-theory, but in general there is no $\C^*$ action on \eqref{ahat}.

The algebraic de Rham complex \eqref{ahat} --- with the de Rham differential --- has been shown by Illusie \cite[Corollary VIII.2.2.8]{Ill} (and more generally Bhatt \cite{Bh}) to be quasi-isomorphic to the pushforward of the \emph{derived de Rham complex} $\Lambda\udot\;\LL_M$ of $M$. (Here $\Lambda\udot$ denotes the alternating sum of \emph{derived} exterior powers.) And, as suggested to us by Ben Antieau, we can prove that the K-theory class of $\Lambda\udot\;\LL_M$ (again \emph{without} its de Rham differential) can be identified with $\Lambda_M$ when $M$ is a local complete intersection.\footnote{When $M$ is not lci $\LL_M$ will have homology in infinitely many degrees in general (even before we take exterior powers), so it is not clear how to define a K-theory class for $\Lambda\udot\LL_M$.}

\begin{prop}\label{ddR} Let $\C^*$ act on $\LL_M$ with weight 1, and suppose $M$ is lci. Then $\Lambda_M=\Lambda\udot\;\LL_M$ in $K_0(M)[\![t]\!]$.
\end{prop}

\begin{proof}
For $M$ lci we have $\LL_{M/A}=I/I^2[1]$ so the exact triangle $\LL_A|_M\to\LL_M\to\LL_{M/A}$ gives, in K-theory,
$$
\LL_M\=\Omega_A\big|_M-I/I^2.
$$
Using the weight one $\C^*$ action on $\LL_M$ this gives
$$
\Lambda\udot\;\LL_M\=\Lambda\udot\;\Omega_A\big|_M\otimes\Sym\udot I/I^2\ \in\ K_0(M)[\![t]\!].
$$
Furthermore $I/I^2$ is locally free so
$$
\Sym^iI/I^2\To I^i/I^{i+1}
$$
is a surjection from a locally free sheaf to a sheaf of the same rank. It is therefore an isomorphism and we have
\[
\Lambda\udot\;\LL_M\=\Lambda\udot\;\Omega_A\big|_M\otimes\Big(\bigoplus\nolimits_{i\ge0}\,I^i\big/I^{i+1}\Big). \qedhere
\]
\end{proof}

By Theorem \ref{thm} this means $\Lambda\udot\;\LL_M$ is in fact polynomial in $t$, so we can set $t=1$ to get a class in non-equivariant K-theory. 

However it does \emph{not} follow (and indeed is not in general true) that the push forward of $\Lambda_M$ can be equated with the algebraic de Rham complex \eqref{ahat} when $M$ is lci. Firstly, the de Rham differential does not preserve the $\C^*$ action we have used, so we cannot lift Illusie's theorem to \emph{equivariant} K-theory. Secondly, this therefore gives us convergence issues; Illusie and Bhatt use the ``Hodge completion" of the derived de Rham complex to get their quasi-isomorphism, and this differs from our completion.

\section{A formula for the virtual structure sheaf}
\label{kFul2}

The foundations of cohomological virtual cycles are laid down in \cite{BF, LT}; we use the notation from \cite{BF}. The foundations for K-theoretic virtual cycles (or ``\emph{virtual structure sheaves}") are laid down in \cite{CFK, FG}; we use the notation from \cite{FG}.

Again let $M$ be a quasi-projective scheme over $\C$. A perfect obstruction theory $E\udot\to\LL_M$ is a 2-term complex of vector bundles $E\udot=\{E^{-1}\to E^0\}$ with a map in $D(\mathrm{Coh}\,M)$ to the cotangent complex\footnote{Or its $\tau^{[-1,0]}$ truncation.} $\LL_M$ which is an isomorphism on $h^0$ and a surjection on $h^{-1}$.

We sometimes call $E\udot$ the \emph{virtual cotangent bundle} $\LL_M^{\!\vir}$ of $M$. Its rank is the \emph{virtual dimension} $\vd:=\rk E^0-\rk E^{-1}$.

By \cite{BF} this data defines a cone $C\subset E_1:=(E^{-1})^*$ from which we may define $M$'s \emph{virtual cycle}
$$
\big[M\big]^{\vir}\ :=\ \iota^![C]\ \in\ A_{\vd}(M),
$$
where $\iota\colon M\to E_1$ is the zero section. Siebert \cite{Sie} proved the alternative formula
$$
[M]^{\vir}\=\left[s\((E\udot)^\vee\)\cap c_F(M)\right]_{\vd}.
$$\smallskip

The K-theoretic analogue of $[M]^{\vir}$ is the \emph{virtual structure sheaf} \cite{FG}
\beq{vss}
\cO_M^{\vir}\ :=\ \big[L\iota^*\,\cO_C\big]\ \in\ K_0(M),
\eeq
where $L\iota^*\,\cO_C$ is a bounded complex because $\iota$ is a regular embedding.
%
%

The construction of Section \ref{kFul} allows us to give a K-theoretic analogue.

\begin{thm} \label{KSie} The virtual structure sheaf \eqref{vss} can be calculated in terms of the K-theoretic Fulton class $\Lambda_M$ \eqref{FuK} as
\beq{formula}
\cO^{\vir}_M\=\left[\frac{\Lambda_M}{\Lambda\udot(E\udot)}\right]_{t=1}\=
\left[\frac{\Lambda_M}{\Lambda\udot\;\LL_M^{\!\vir}}\right]_{t=1}.
\eeq
\end{thm}

In particular, when $M$ is smooth $\Lambda_M=\Lambda\udot\;\Omega_M$ and $[E\udot]=\Omega_M-\mathrm{ob}_M^*$, so \eqref{formula} recovers $\cO^{\vir}=\Lambda\udot\mathrm{ob}_M^*$.

\begin{proof}
By \cite{BF}, the perfect obstruction theory $E\udot\to\LL_M$ induces a cone $C\ \subset\ E_1$
which Siebert \cite[proof of Proposition 4.4]{Sie} shows\footnote{After possibly replacing $E\udot$ by a quasi-isomorphic 2-term complex of vector bundles.} sits inside an exact sequence of cones
\beq{SSES}
0\To T_A\big|_M\To C_MA\oplus E_0\To C\To 0.
\eeq
Here $A$ is any smooth ambient space containing $M$ with ideal $I$, so that $C_MA=\Spec\bigoplus_{i\ge0}I^i/I^{i+1}$.

As before we give the $E_i$ and $T_A|_M$ weight $-1$ (so $E^i$ and $\Omega_A|_M$ have weight $1$) and let $\iota\colon M\into E_1$ and $\pi\colon E_1\to M$ denote the zero section and projection respectively. Then\vspace{-3mm}
$$
\cO^{\vir}\,=\,L\iota^*\cO_C\,=\,\pi_*\iota_*L\iota^*\cO_C\,=\,\pi_*\(\cO_C\stackrel L\otimes\iota_*\;\cO_M\)\,=\,\pi_*\(\cO_C\otimes\Lambda\udot E^{-1}\)
$$
evaluated at $t=1$, by \eqref{vss} and the Koszul resolution of $\iota_*\;\cO_M$. By \eqref{SSES},
$$
\pi_*\;\cO_C\=\Sym\udot E^0\otimes\Big(\bigoplus\nolimits_{i\ge0}\,I^i\big/I^{i+1}\Big)\otimes\Lambda\udot\Omega_A\big|_M\;,
$$
so by \eqref{FuK},
\[
\cO^{\vir}\=\left[\Lambda_M\otimes\Sym\udot E^0\otimes\Lambda\udot E^{-1}\right]_{t=1}\=\left[\Lambda_M\big/\Lambda\udot(E\udot)\right]_{t=1}. \qedhere
\]
\end{proof}

\begin{cor} $\cO_M^{\vir}$ depends only on $M$ and the K-theory class of $E\udot$.
\end{cor}

Of more interest in enumerative K-theory is the \emph{twisted} virtual structure sheaf
$$
\widehat\cO^{\;\vir}_{\!M}=\cO_M^{\vir}\otimes\det(E\udot)^{1/2}
$$
of Nekrasov-Okounkov \cite{NO}. Here we twist by a choice of square root of the virtual canonical bundle $\det(E\udot)=\det E^0\otimes(\det E^{-1})^*$. The above shows it depends only on $M$, the K-theory class of $E\udot$, and the choice of square root (or ``\emph{orientation data}").

\section{K-theoretic Chern classes}\label{KCh}

The K-theory of complex vector bundles is an oriented cohomology theory, so admits a notion of Chern classes. In particular the $r$th Chern class of a rank $r$ bundle $E$ --- the K-theoretic Euler class of $E$ --- is
\beq{KEuler}
c_r^K(E)\ :=\ \Lambda\udot E^*\ \in\ K^0(M).
\eeq
Whenever there exists a transverse section of $E$
this is the class $[\cO_Z]$ of the structure sheaf of its zero locus $Z\subset M$.

Dave Anderson asked if the K-theoretic Fulton class $\Lambda_M$ of \eqref{FuK} could be written in terms of K-theoretic Chern classes and Segre classes.\footnote{See \cite{HIMN, A} for one definition of K-theoretic Segre classes.} I think there is probably no simple formula along these lines, essentially because the formal group law for K-theory is nontrivial, complicating the K-theoretic analogue of the formula \eqref{ce} below. For instance if we simply substitute $c^K$ and $s^K$ into the right hand side of the definition \eqref{FHstar}
\beq{FH*}
c\_F(M)\ :=\ c\(T_A|_M\)\cap s(C_{M/A})\ \in\ A_\bullet(M)
\eeq
we do \emph{not} get $\Lambda_M$.

However if we first note the usual Fulton class $c\_F(M)$ can be re-expressed in terms of a $\C^*$-equivariant Euler class, then there is a nice analogue in $\C^*$-equivariant K-theory. Begin by rewriting \eqref{FH*} as
\beq{upmss}
c\_F(M)\=\left[t^dc_{\frac1t}\(T_A|_M\)\cdot t^{m-d}s_{\frac1t}(C_{M/A})\right]_{t=1}
\eeq
where $m=\dim M,\ d=\dim A$ and $c_t:=\sum t^ic_i,\ s_t:=\sum t^is_i$. Now let $\t$ denote the standard weight one irreducible representation of $\C^*$, so we can recycle the notation $t$ as $t:=c_1(\t)\in H^2(B\C^*)$. This allows us to rewrite the first term in the brackets of \eqref{upmss} as the equivariant Euler class
\beq{ce}
t^dc_{\frac1t}\(T_A|_M\)\=e^{\C^*}(T_A|_M\otimes\t)
\eeq
in the $\C^*$-equivariant cohomology,\footnote{As usual we localise by inverting $t$.} or Chow cohomology, of $M$ (with the trivial $\C^*$ action). By \eqref{KEuler} its K-theoretic analogue is
\beq{first}
t^dc_{\frac1t}\(T_A|_M\)\ \longleftrightarrow\ \Lambda\udot(T_A^*|\_M\otimes\t^{-1}).
\eeq
For the same reason I would like to think of the second term in the brackets of \eqref{upmss} as an \emph{inverse} equivariant Euler class of $C_{M/A}$,
\beq{segE}
t^{m-d}s_{\frac1t}(C_{M/A})\ ``="\ \frac1{e^{\C^*}(C_{M/A}\otimes\t)}\,.
\eeq
When $C_{M/A}$ is a bundle this makes perfect sense, at least. In that case the K-theoretic analogue of \eqref{segE} is, by \eqref{KEuler},
$$
\frac1{\Lambda\udot(C_{M/A}^*\otimes\t^{-1})}\=\Sym\udot(C_{M/A}^*\otimes\t^{-1})\=\bigoplus_{i\ge0}\t^{-i}(I^i/I^{i+1}),
$$
after completing with respect to $\t$ or localising appropriately. So it is natural to think of this as giving the K-theoretic analogue
\beq{second}
t^{m-d}s_{\frac1t}(C_{M/A})\ \longleftrightarrow\ \bigoplus\nolimits_{i\ge0}\,\t^{-i}(I^i\big/I^{i+1})
\eeq
even when $C_{M/A}$ is not a bundle. Substituting (\ref{first}, \ref{second}) into \eqref{upmss} gives
$$
\Lambda\udot(\Omega_A|_M\otimes\t^{-1})\otimes\bigoplus\nolimits_{i\ge0}\,\t^{-i}(I^i\big/I^{i+1}).
$$
Up to replacing $\t$ by $\t^{-1}$ this is precisely our K-theoretic Fulton class \eqref{FuK}.

\bibliographystyle{halphanum}
\bibliography{references}

\bigskip \noindent {\tt{richard.thomas@imperial.ac.uk}} \medskip

\noindent Department of Mathematics \\
\noindent Imperial College London\\
\noindent London SW7 2AZ \\
\noindent United Kingdom

\end{document}